\documentclass[reqno]{amsart}
\usepackage[numbers,square,comma,sort&compress]{natbib}
\usepackage{amsmath,amssymb}
\usepackage{color}

\newcommand{\real}{\mathbb{R}}
\newcommand{\nat}{\mathbb{N}}

\newcommand{\rn}{\real^N}
\newcommand{\intrn}{\int_{\real^N}}

\newcommand{\R}{\mathbb{R}}

\newcommand{\eps}{\varepsilon}

\newcommand{\de}{\delta}
\newcommand{\ffi}{\varphi}

\newcommand{\lam}{\lambda}
\newcommand{\al}{\alpha}

\newcommand{\wto}{\rightharpoonup}
\newcommand{\la}{\langle}
\newcommand{\ra}{\rangle}
\newcommand{\wt}{\widetilde}
\newcommand{\diff}{\,\mathrm{d}}
\newcommand{\dif}{\mathrm{d}}
\newcommand{\Ve}{\Vert}
\newcommand{\ve}{\vert}
\renewcommand\emptyset{\mbox{\Large \o}}
\newcommand{\sm}{\setminus}

\newcommand{\leqs}{\leqslant}
\newcommand{\geqs}{\geqslant}

\let\le=\leqslant
\let\ge= \geqslant

\newcommand{\calF}{\mathcal{F}}

\newcommand{\disp}{\displaystyle}

\DeclareMathOperator \di{div}

\DeclareMathOperator \dist{dist}
\DeclareMathOperator* \essinf{ess\,inf}
\DeclareMathOperator* \esssup{ess\,sup}

\newtheorem{theorem}{Theorem}
\newtheorem{lemma}[theorem]{Lemma}
\newtheorem{proposition}[theorem]{Proposition}
\newtheorem{corollary}[theorem]{Corollary}
\newtheorem{remark}[theorem]{Remark}

\begin{document}

\title{Existence of nodal solutions for quasilinear elliptic problems in $\rn$}

\author{Ann Derlet}

\address{Haute Ecole Robert Schuman\\
Rue de Mageroux, 27\\
B-6760 Virton, Belgique}

\email{aderlet@hers.be}

\author{Fran\c cois Genoud}

\address{Faculty of Mathematics\\
University of Vienna\\
Oskar-Morgenstern-Platz 1\\
1090 Vienna, Austria}

\email{francois.genoud@univie.ac.at}


\thanks{The research work reported in this paper was initiated while A.D. was
at the CeReMath, Universit\'e Toulouse 1 Capitole, and F.G. was at
Heriot--Watt University. A.D. visited Heriot--Watt once and
F.G. visited the CeReMath twice during this work. We are grateful to both 
institutions for their hospitality. 
We also thank the anonymous referee for helpful comments,
Charles Stuart for insightful discussions, and
Peter Tak\'a\v c for indicating useful references to us.
F.G. acknowledges the support of the U.K.
Engineering and Physical Sciences Research Council [EP/H030514/1],
and of the ERC Advanced Grant ``Nonlinear studies of water
flows with vorticity''.}


\begin{abstract}
We prove the existence of one positive, one negative, and one sign-changing
solution of a $p$-Laplacian equation on $\rn$, with a $p$-superlinear subcritical term.
Sign-changing solutions of quasilinear elliptic equations set on the whole of $\rn$ 
have only been scarcely investigated in the literature. 
Our assumptions here are similar to those previously used by some authors 
in bounded domains, and
our proof uses fairly elementary critical point theory,
based on constraint minimization on the nodal Nehari set. 
The lack of compactness due to the unbounded domain is overcome by working
in a suitable weighted Sobolev space.
\end{abstract}

\maketitle

\section{Introduction}

In this paper we study sign-changing solutions of the quasilinear elliptic equation
\begin{equation}\label{problem}
-\Delta_p u = \lambda A(x) |u|^{p-2}u + g(x,u) \quad\text{on} \ \R^N,
\end{equation}
where $\Delta_p u=\di(|\nabla u|^{p-2}\nabla u)$ with $1<p<N$, $N\ge1$, 
and $\lam$ is a real parameter. We suppose that the coefficient $A:\rn\to\real$ satisfies:

\medskip
\noindent
$\boldsymbol{(A1)}$ $A$ is measurable, with $A>0$ a.e.,
$A \in L^{\infty}(\R^N)$ and $A^{-\frac{1}{p-1}}\in L^1_\mathrm{loc}(\rn)$. 

\medskip
\noindent
Note that the last condition in $(A1)$ holds if, for instance, 
$\essinf_{\Omega}A>0$ for any bounded open set $\Omega\subset\rn$.
An additional condition relating $A$ and $\lam$, formulated in 
Section~\ref{tuning.sec} below, will also play a crucial role in our analysis.

\noindent
We suppose that the function $g$ satisfies the following assumptions, 
where $p^*=\frac{Np}{N-p}$:

\medskip
\noindent
$\boldsymbol{(g1)}$ $g:\R^N\times\R\to\R$ is a Carath\'eodory function, 
$g(x,\cdot) \in C^1(\real)$ for a.e. $x \in \real$, and there exist $q\in(p,p^*)$ 
and $B\in L^{\frac{p^*}{p^*-q}}(\R^N) \cap L^\infty(\R^N)$, $B>0$ a.e., such that
\[
|g_s(x,s)| \le B(x) |s|^{q-2}, \quad\text{for a.e.} \ x\in\R^N,\ s\in\R.
\]

\medskip 
\noindent
$\boldsymbol{(g2)}$ There exist $\theta>p$ and $R>0$ such that
\[
g(x,s)s\geqs \theta\,G(x,s)>0, \quad \text{for a.e.} \ x\in\R^N, \ |s|\geqs R,
\]
where $G(x,s)=\int_0^s g(x,t)\diff t$.

\smallskip
\noindent
$\boldsymbol{(g3)}$ The mapping $s\mapsto\disp\frac{g(x,s)}{|s|^{p-1}}$ is strictly increasing in 
$s\in\R\setminus\{0\}$, for a.e. $x\in\R^N$.

\medskip
Assumptions $(g1)$ and $(g2)$ are often referred to by saying that $g$ is {\em subcritical} and 
{\em $p$-superlinear}\footnote{The term `$(p-1)$-superhomogeneous' is sometimes used instead
of $p$-superlinear.}, respectively.
The hypothesis $B\in L^\infty(\R^N)$ will be convenient to establish compactness of a weighted
Sobolev embedding which plays a important role in the paper. However, it can be relaxed to
a local integrability assumption, namely $B\in L^{s}_\mathrm{loc}(\R^N)$ for some
large enough $s$ (see Remark~\ref{localintegr.rem} in
Appendix~\ref{compact.sec} for a more a precise statement).
Our assumptions on $B$ imply that \eqref{problem} is a compact perturbation
(in a sense that will be made precise in Lemma~\ref{compcont.lem})
of the $p$-linear eigenvalue problem
\begin{equation}\label{eigenproblem}
-\Delta_p u = \lambda A(x) |u|^{p-2}u \quad\text{on} \ \R^N.
\end{equation}
The existence of eigenvalues and 
eigenfunctions for this problem has been discussed by Allegretto and Huang \cite{AlHu}, 
even in the case of an indefinite weight $A$ (i.e. when $A$ does not have a constant sign), 
as long as $A$ is positive on a set of positive measure and satisfies an integrability condition
at infinity. Bifurcation of solutions of \eqref{problem} from the principal eigenvalue of
\eqref{eigenproblem} has been studied by Dr\'abek and Huang \cite{DraHu}, who obtained
global continua of positive and negative solutions of \eqref{problem}.
Bifurcation from higher eigenvalues 
--- which would provide sign-changing solutions --- 
is a difficult problem. In fact, sign-changing solutions for quasilinear equations in the whole of
$\rn$ have only been scarcely investigated; see however \cite{chen}, where a problem with 
cylindrical symmetry is considered, and \cite{guo}, dealing with a $p$-asymptotically
linear problem. 

On the other hand, there is a fair amount of literature
on solutions of quasilinear elliptic equations in bounded domains.
In the radial case, the Dirichlet problem in a ball has been solved by 
Del Pino and Man\'asevich \cite{dm} by the bifurcation approach, yielding infinitely many
nodal\footnote{Throughout the paper we will use the terms `nodal' and 
`sign-changing' interchangeably.} solutions. In the non-radial setting, an important contribution
(probably the most general so far) is due to Bartsch et al.~\cite{BaLiuWe}, 
who proved existence and multiplicity of nodal solutions
for $p$-Laplacian Dirichlet problems in smooth bounded
domains. Their results are obtained by critical point theory in Banach spaces, 
making clever use of a suitable pseudo-gradient flow. In fact in $\rn$, using similar arguments,
a sign-changing solution of an auxiliary $p$-superlinear problem (denoted $(P)_\mu$)
is obtained in the course of the proof in \cite{guo}. However, the equation $(P)_\mu$
considered there has a different structure from ours, mainly due to the $p$-asymptotically
linear nature of the original problem.

It is worth remarking that the hypotheses $(\text{H}_0)$--$(\text{H}_3)$ 
used to prove the existence of a nodal solution in Bartsch et al.~\cite{BaLiuWe}
are structurally similar to ours. However
in the present work we use a more elementary method, 
based on constraint minimization
on the `nodal Nehari set'. This approach
originated in a series of works on the Dirichlet problem for semilinear equations,
a review of which can be found in \cite{BaWaWi}. It has also been used more recently  
in the context of a prescribed mean-curvature problem in Bonheure et al.~\cite{anndenis},
from which some of our arguments are inspired.
Note that the Nehari approach stronlgy relies on the monotonicity assumption $(g3)$,
which is not needed in \cite{BaLiuWe}.
We will focus here on the case of a positive coefficient $A$ and,
under the hypotheses $(A1)$, $(g1)$--$(g3)$,
we will give a sufficient condition relating $\lam$ and $A$ 
for the existence of at least one positive, one negative, 
and one nodal solution of \eqref{problem}. This condition  --- assumption $(A,\lam)$ 
below --- should be compared with hypothesis $(\text{H}_3)$ in \cite{BaLiuWe},
involving the first eigenvalue of the $p$-linear problem.
We will also deduce corresponding existence results for the problem
\begin{equation}\label{problem2}
-\Delta_p u = \wt A(x) |u|^{p-2}u + g(x,u) \quad\text{on} \ \R^N,
\end{equation}
under appropriate conditions on the coefficient $\wt A \in L^{\infty}(\R^N)$. 

Our approach is based on a variational
formulation of \eqref{problem} in a weighted Sobolev space. 
More precisely, let us define the norm
$\Ve\cdot\Ve_{W_A}$ on $C_0^\infty(\rn)$ by
\begin{equation}\label{normA}
\Ve u\Ve_{W_A}=\Big(\intrn |\nabla u|^p + A(x) |u|^p \diff x\Big)^{1/p}.
\end{equation}
We will work in the space $W_A(\rn)=\overline{C_0^\infty(\rn)\,}^{\Ve\cdot\Ve_{W_A}}$.
Some elementary properties of $W_A(\rn)$ will be established in Section~\ref{prelim.sec}. 
In particular, $W_A(\rn)$ is a separable and uniformly convex (hence reflexive) Banach space. 
Since $A\in L^{\infty}(\R^N)$, $W_A(\rn) \supseteq W^{1,p}(\rn)$, 
with equality if $A$ is bounded away from zero. 
We will often merely write $W_A$ for $W_A(\rn)$, although some properties
of $W_A(\Omega)$ will be given in Section~\ref{prelim.sec}, for more 
general open subsets $\Omega\subset\rn$.

With $G$ defined as in $(g2)$, we introduce the functional $S_\lam : W_A \to \real$,
\begin{equation}\label{functional}
S_\lam(u)=\frac1p\intrn |\nabla u|^p -\lam A(x) |u|^p \diff x - \intrn G(x,u) \diff x.
\end{equation}
Due to $(g2)$ $S_\lam$ is not coercive, so one cannot apply 
the direct method of the calculus of variations to find critical points of $S_\lam$. 
We will see that constraint minimization on Nehari-type sets provides an efficient alternative
to obtain critical points, and to discuss their nodal properties.


\subsection{Main results}\label{tuning.sec}
We now formulate the assumption
relating $A$ and $\lam$ that we shall use to prove our results. 
For $A$ satisfying $(A1)$, let
\begin{equation}\label{lambdaA}
\lam_A=\inf_{u\in W_A\sm\{0\}}\frac{\intrn |\nabla u|^p \diff x}{\intrn A(x)|u|^p \diff x}\ge 0.
\end{equation}
Provided $A$ and $\lam$ satisfy

\medskip
\noindent
$\boldsymbol{(A,\lam)}$ $\lam<\lam_A$,

\medskip
\noindent
it will be shown in Section~\ref{prelim.sec} that
\begin{equation}\label{normlambda}
\Ve u\Ve_{\lam}=\Big(\intrn |\nabla u|^p -\lam A(x) |u|^p \diff x\Big)^{1/p}
\end{equation}
defines a (quasi)norm\footnote{see Lemma~\ref{normequ.lem} and Remark~\ref{normequ.rem}} 
on $W_A$, which is equivalent to $\Ve\cdot\Ve_{W_A}$.
This enables one to rewrite $S_\lam : W_A \to \real$ as
\begin{equation}\label{refunctional}
S_\lam(u)=\frac1p\Ve u\Ve_{\lam}^p - \intrn G(x,u) \diff x,
\end{equation}
which is a convenient way to exhibit the main properties of the $p$-homogeneous
part of $S_\lam$ with respect to the variational procedure. Note that if
$\lam<0$ then $(A,\lam)$ is always satisfied, and $\Ve\cdot\Ve_{\lam}$ is a norm, which 
is equivalent to $\Ve\cdot\Ve_{W_A}$ on $W_A$. The more interesting case of positive
$\lam$ requires $\lam_A>0$. As will be shown in Section~\ref{existence.sec},
a sufficient condition for this to hold is

\medskip
\noindent
$\boldsymbol{(A2)}$ $A\in L^{N/p}(\rn)$.

\medskip
\noindent

We shall see in Lemma~\ref{derivative.lem} that $S_\lam\in C^1(W_A,\real)$ provided
$A$ satisfies $(A_1)$. Our main results concern weak solutions of problems \eqref{problem}
and \eqref{problem2}. A function $u\in W_A$ is called a {\em solution} 
of \eqref{problem} if and only if $S_\lam'(u)=0$, that is, if and only if
\begin{equation*}\label{weak}
\intrn |\nabla u|^{p-2}\nabla u\cdot\nabla\ffi \diff x
-\lam\intrn A(x)|u|^{p-2}u\ffi \diff x -\intrn g(x,u)v \diff x=0 \ \forall\,\ffi\in W_A,
\end{equation*}
with a similar definition for solutions of \eqref{problem2}.
The existence of weak solutions will be proved by a variational approach in 
Section~\ref{existence.sec}, under hypotheses $(A1)$,
$(A,\lam)$ and $(g1)$--$(g3)$. In addition to ensuring that $\lam_A>0$, assumption
$(A2)$ is also needed to obtain extra regularity of the solutions.
Our main result is the following.

\begin{theorem}\label{main.thm}
Suppose that the hypotheses $(A1)$, $(A2)$, 
$(A,\lam)$ and $(g1)$--\,$(g3)$ are satisfied. Then there 
exist three solutions $u_1,u_2,u_3\in C^{1,\alpha}_\mathrm{loc}(\rn)$ of \eqref{problem}, 
with $u_1>0$, $u_2<0$, and $u_3^\pm\not\equiv 0$. 
Furthermore, $u_3$ has exactly two nodal domains. 
\end{theorem}

The definition of the class $C^{1,\alpha}_\mathrm{loc}(\rn)$ is recalled in Section~\ref{reg.sec}.
As usual, we denote by $u^\pm$ the positive and negative parts of $u$. More precisely, we use 
here the convention $u^\pm(x):=\pm\max\{\pm u(x),0\}$, so that $u=u^++u^-$.
For a continuous function $u:\rn\to\real$, a {\em nodal domain} is a 
connected component of $\rn\setminus u^{-1}(\{0\})$.

It is also possible to formulate a version of our results from which the parameter $\lam$ is 
absent, pertaining to problem \eqref{problem2}. 

\begin{corollary}\label{main.cor}
Suppose that $(g1)$--\,$(g3)$ are satisfied.
If either
\begin{itemize}
\item[(i)] $\wt A$ satisfies $(A1)$, $(A2)$ and $\lam_{\wt A}>1$,
\item[(ii)] or $-\wt A$ satisfies $(A1)$ and $(A2)$,
\item[(iii)] or $\wt A=0$ a.e.,
\end{itemize}
then \eqref{problem2} has three solutions in $C^{1,\alpha}_\mathrm{loc}(\rn)$.
Furthermore, the solutions have the same nodal structure as in Theorem~\ref{main.thm},
i.e. one is positive, one is negative, and one is sign-changing with two nodal domains.
\end{corollary}

\begin{proof}
In case (i) holds, the result follows from Theorem~\ref{main.thm} by 
letting $A=\wt A$ and $\lam=1<\lam_{\wt A}$.
Case (ii) follows by letting $A=-\wt A$ and $\lam=-1$, while for (iii) one can
take any $A$ satisfying $(A1)$--$(A2)$, and $\lam=0$.
\end{proof}

The rest of the paper is organized as follows. In Section~\ref{prelim.sec} 
we first establish some properties of the spaces $W_A(\Omega)$,
in particular a compact embedding that plays a central role in the proof of our main result.
The existence of three critical points is proved in Section~\ref{existence.sec} by minimization
on Nehari-type sets. The regularity of solutions is discussed in Section~\ref{reg.sec}, where the
proof of Theorem~\ref{main.thm} is then completed.

We shall use the letter $C$ (possibly with an index) to denote various positive constants,
the exact value of which is not relevant to the analysis.


\section{Preliminaries}\label{prelim.sec}

We start this section by some preliminary results about the functional setting, 
which will be useful to our existence theory. In particular we establish embedding properties 
of the space $W_A(\Omega)$ which will play an important role in our analysis. For an
arbitrary open set $\Omega\subset\rn$, let
\[
W_A(\Omega)=\{u\ve_\Omega : u\in W_A(\rn)\},
\]
where $W_A(\rn)$ has been defined in the introduction.
Given a positive measurable function $B:\Omega\to\real_+$ and $q\ge 1$,
we define the weighted Lebesgue space 
\[
L^q_B(\Omega)=\{u:\Omega\to\real : u \ \text{is measurable and} \ 
\int_\Omega |u|^q B(x) \diff x<\infty\},
\]
endowed with its natural norm 
$\Ve u\Ve_{L^q_B(\Omega)}=(\int_\Omega |u|^q B(x) \diff x)^{1/q}$. 
When there is no risk of confusion, we shall merely write $W_A$ and $L^q_B$ for 
$W_A(\Omega)$ and $L^q_B(\Omega)$. 
We may also use the shorthand notation $\Ve\cdot\Ve_r\equiv \Ve\cdot\Ve_{L^r}$
for the usual Lebesgue norms.

Under appropriate assumptions, we will show that $W_A(\Omega)$ is continuously and
compactly embedded into $L^q_B(\Omega)$. 
Our proof will rely on the classic theory of Sobolev spaces as 
presented for instance in Brezis \cite[Chapitre~IX]{Br}.\footnote{Note that an 
English translation of
Brezis' book is also available \cite{BrE}. However, in the original (French) version 
\cite{Br}, the compactness result we shall use is formulated in a way which is slightly 
better suited to our proof of Proposition~\ref{sobolev.prop}, 
so we will rather refer to \cite{Br} throughout.}
But let us first state the following elementary properties.

\begin{proposition}\label{bspace.prop}
Let $A$ satisfy $(A1)$ and $B:\Omega\to\real$ be measurable with $B>0$ a.e..
\begin{itemize}
\item[(i)] $W_A(\Omega)$ is a separable, uniformly convex --- hence reflexive --- Banach space 
satisfying $W^{1,p}(\Omega) \subseteq W_A(\Omega)$, with equality if $\essinf_{\rn}A>0$.
\item[(ii)] For any $1<q<\infty$, $L^q_B(\Omega)$ is a separable reflexive Banach space.
\end{itemize}
\end{proposition}

\begin{proof}
The fact that $W_A(\rn)=W^{1,p}(\rn)$ if $\essinf_{\rn}A>0$ follows from the definition of $W_A(\rn)$. 
The remaining statements can be found in \cite{BenFor} and \cite{DraKufNic}. 
\end{proof}

\begin{proposition}\label{sobolev.prop}
Suppose that $\Omega\subset\rn$ is either $\rn$ or a bounded open set with $C^1$ boundary.
Let $q\in(p,p^*)$ and $B\in L^{\frac{p^*}{p^*-q}}(\rn)\cap L^\infty(\rn)$. Then we have
$W_A(\Omega) \subset L_B^q(\Omega)$ and there is a constant $C>0$ such that
\begin{equation}\label{sobolev}
\Ve u\Ve_{L^q_B(\Omega)} \le C \Ve u \Ve_{W_A(\Omega)}, \quad u\in W_A(\Omega).
\end{equation}
Furthermore, the embedding is compact.
\end{proposition}

\begin{proof}
To prove (i), let us first consider the case where $\Omega=\rn$. Notice that 
$W^{1,p}(\rn)$ is a dense subspace of $W_A(\rn)$.
We shall thus start by proving \eqref{sobolev} for $u\in W^{1,p}(\rn)$, and then argue by
density. For $u\in W^{1,p}(\rn)$, it follows from H\"older's inequality and
the Sobolev embedding theorem that
\begin{equation*}
\intrn |u|^q B(x) \diff x \le \Ve B\Ve_{\frac{p^*}{p^*-q}} \Ve u \Ve_{p^*}^q 
\le \Ve B\Ve_{\frac{p^*}{p^*-q}} C \Ve \nabla u \Ve_{p}^q
\le C \Ve B\Ve_{\frac{p^*}{p^*-q}} \Ve u \Ve_{W_A}^q,
\end{equation*}
and so $\Ve u\Ve_{L_B^q} \le C\Ve u\Ve_{W_A}$ for all $u\in W^{1,p}(\rn)$. 
Now for any $u\in W_A(\rn)$,
there is a sequence $(u_n)\subset W^{1,p}(\rn)$ such that $u_n\to u$ in $W_A(\rn)$. 
But for each $n$,
there holds $\Ve u_n\Ve_{L_B^q} \le C\Ve u_n\Ve_{W_A}$, so passing to the limit using Fatou's
lemma yields \eqref{sobolev}. The case where $\Omega$ is a bounded domain with smooth 
boundary follows from the case $\Omega=\rn$ by adapting the extension theorem 
\cite[Th\'eor\`eme~IX.7]{Br} to the present context.
The compactness of the embedding is proved in Appendix~\ref{compact.sec}.
\end{proof}

\begin{remark}\label{classicsobolev.rem}
\rm
Observe that, by density, 
the classic Sobolev inequality, $\Ve u\Ve_{p^*}\le C\Ve\nabla u\Ve_p$ 
for $u\in W^{1,p}(\Omega)$,
extends to $u\in W_A(\Omega)$, so that 
\begin{equation}\label{xtrmsobolev}
\Ve u\Ve_{L^{p^*}(\Omega)}\le C\Ve\nabla u\Ve_{L^p(\Omega)}
\le C\Ve u\Ve_{W_A(\Omega)}, \quad u\in W_A(\Omega).
\end{equation}
\end{remark}

We are now in a position to prove that the functional $S_\lam: W_A(\rn) \to \real$ 
defined in \eqref{functional}
is of class $C^1$. We shall denote by $\la\cdot,\cdot\ra:W_A^*\times W_A\to\real$ the 
duality pairing between $W_A$ and its topological dual $W_A^*$. 

\begin{lemma}\label{derivative.lem}
Let $A$ satisfy assumption $(A1)$.
Then we have $S_\lam \in C^1(W_A(\rn),\real)$ and, for all $u,v \in W_A(\rn)$,
\begin{equation}\label{derivative}
\la S'_\lam(u),v \ra = \intrn |\nabla u|^{p-2}\nabla u \cdot \nabla v -\lam A(x) |u|^{p-2}uv \diff x
- \intrn g(x,u)v \diff x.
\end{equation}
\end{lemma}

\begin{proof}
The proof of Lemma~\ref{derivative.lem} follows from the continuity of the embedding in 
Proposition~\ref{sobolev.prop}. 
To avoid disrupting the exposition with technicalities, we postpone 
it to Appendix~\ref{derivative.sec}.
\end{proof}

\begin{lemma}\label{compcont.lem}
Under assumptions $(g1)$, $(A1)$ and $(A2)$, the following statements hold.
\begin{itemize}
\item[(i)] The functionals $W_A(\rn)\to\real$, 
$u\mapsto\intrn g(x,u)u\diff x$, 
$u\mapsto\intrn G(x,u) \diff x$
are compact, in the sense that they 
map bounded sequences to relatively compact ones.
\item[(ii)] The functional $W_A(\rn)\to\real$,
$u\mapsto\intrn A(x)|u|^p \diff x$
is also compact.
\end{itemize}
\end{lemma}

\begin{proof}
(i) Consider a bounded sequence $(u_n)\subset W_A$. By Proposition~\ref{sobolev.prop}
$(u_n)$ is bounded in $L_B^q(\rn)$, and
there exist a subsequence (still denoted by $(u_n)$) and an element
$u \in L_B^q(\rn)$ such that $u_n\to u$ in $L_B^q(\rn)$.
It follows from $(g1)$ that
\begin{align*}\label{estphi}
|\Phi(u_n)-\Phi(u)|
&\le \intrn |g(x,u_n)u_n-g(x,u)u|\diff x \\ 
&\le \intrn  |g(x,u_n)-g(x,u)||u_n|\diff x+\intrn  |g(x,u)||u_n-u|\diff x \\
&\le \intrn  B\big||u_n|^{q-2}u_n-|u|^{q-2}u\big||u_n|\diff x
+\intrn B |u|^{q-1}|u_n-u|\diff x,
\end{align*}
where, by H\"older's inequality,
\begin{multline*}
\intrn  B\big||u_n|^{q-2}u_n-|u|^{q-2}u\big||u_n|\diff x
\le C\Big(\intrn B\big||u_n|^{q-2}u_n-|u|^{q-2}u\big|^\frac{q}{q-1}\diff x\Big)^\frac{q-1}{q}.
\end{multline*}
Since $u_n\to u$ in $L_B^q(\rn)$, we can suppose that $u_n\to u$ pointwise a.e., and that
$B|u_n|^q\le f$, uniformly in $n$, for some $f\in L^1(\rn)$. It then follows by dominated 
convergence that the right-hand side of the above inequality goes to zero as $n\to\infty$. 
On the other hand, by H\"older's inequality,
\begin{equation*}\label{estphi2}
\intrn  B |u|^{q-1}|u_n-u|\diff x \le\Ve u\Ve_{L_B^q}^{q-1}\Ve u_n-u\Ve_{L_B^q}\to 0
\quad\text{as} \ n\to\infty,
\end{equation*}
which concludes the proof. A similar argument shows that $u\mapsto\intrn G(x,u) \diff x$ 
is compact.

(ii) Consider again a bounded sequence $(u_n)\subset W_A$. 
By Proposition~\ref{bspace.prop}~(i) there exists a subsequence
(still denoted by $(u_n)$) and an element
$u \in W_A$ such that $u_n\wto u$ weakly in $W_A$.
By H\"older's inequality we have, for any open set $\Omega\subset\rn$,
\[
\int_\Omega A(x) \big| |u_n|^p - |u|^p\big| \diff x \leqs 
\Ve A\Ve_{L^{r}(\Omega)} \big\Ve  |u_n|^p - |u|^p \big\Ve_{L^{s}(\Omega)},
\]
for some $r>\frac{N}{p}$ and $s<\frac{N}{N-p}$.
Since $u\mapsto |u|^p$ is continuous from $L^{ps}(\Omega)$ to 
$L^{s}(\Omega)$ and, for $\Omega$ bounded, 
the embedding $W_A(\Omega)\subset L^{ps}(\Omega)$
is compact, it follows that
\[
\int_\Omega A(x) \big| |u_n|^p - |u|^p\big| \diff x\to 0, \quad \Omega \ \text{bounded}.
\]
On the other hand,
\begin{align*}
\int_{\rn\setminus\Omega}A(x) \big| |u_n|^p - |u|^p\big| \diff x	&\leqs
\big\Ve  |u_n|^p - |u|^p \big\Ve_{L^{\frac{N}{N-p}}(\rn\setminus\Omega)}
\Ve A\Ve_{L^{N/p}(\rn\setminus\Omega)} \\ 
&\leqs (\Ve u_n\Ve_{\frac{Np}{N-p}}^p+\Ve u\Ve_{\frac{Np}{N-p}}^p)\Ve 
A\Ve_{L^{N/p}(\rn\setminus\Omega)} \\
&\leqs
(\Ve \nabla u_n\Ve_p^p+\Ve \nabla u\Ve_p^p)\Ve 
A\Ve_{L^{N/p}(\rn\setminus\Omega)}\\
&\leqs C\Ve A\Ve_{L^{N/p}(\rn\setminus\Omega)},
\end{align*}
which can be made arbitrarily small by choosing $|\Omega|$ large enough.
\end{proof}

We conclude this section by showing that, under hypothesis $(A,\lam)$, the (quasi)norm 
$\Ve\cdot\Ve_\lam$ defined in \eqref{normlambda} is equivalent to $\Ve\cdot\Ve_{W_A}$.

\begin{lemma}\label{normequ.lem}
Let $A$ and $\lam$ satisfy the hypotheses $(A1)$ and $(A,\lam)$. Then there exist
constants $c_i=c_i(\lam)>0, \ i=1,2,$ such that
\begin{equation}\label{normequ}
c_1\Ve u\Ve_{W_A} \le \Ve u\Ve_\lam \le c_2\Ve u\Ve_{W_A}, \quad u \in W_A(\rn).
\end{equation}
\end{lemma}

\begin{proof}
The second inequality follows directly from the definition of $\Ve\cdot\Ve_\lam$ 
in \eqref{normlambda}, with $c_2=(\max\{1,|\lam|\})^{1/p}$. It actually holds for
any $\lam\in\real$. The condition $(A,\lam)$ is required to prove the first inequality 
in \eqref{normequ}, which we do now. 
Let $\eps>0$. By the definition of $\lam_A$ in \eqref{lambdaA} we have, for any $u \in W_A$,
\begin{align*}
\Ve u\Ve_\lam^p &= \eps\intrn |\nabla u|^p \diff x + (1-\eps)\intrn |\nabla u|^p \diff x
					-\lam \intrn A(x) |u|^p \diff x\\
				&\ge \eps\intrn |\nabla u|^p \diff x + (1-\eps)\lam_A\intrn A(x) |u|^p \diff x
					-\lam \intrn A(x) |u|^p \diff x\\
				&=\eps\intrn |\nabla u|^p \diff x 
					+ [(1-\eps)(\lam_A-\lam)-\eps\lam]\intrn A(x) |u|^p \diff x\\
				&\ge\min\{\eps,[(1-\eps)(\lam_A-\lam)-\eps\lam]\}\Ve u\Ve_{W_A}^p.
\end{align*}
Since $(1-\eps)(\lam_A-\lam)-\eps\lam\to\lam_A-\lam>0$ as $\eps\to0$, we can choose
$\eps>0$ such that $c_1:=(\min\{\eps,[(1-\eps)(\lam_A-\lam)-\eps\lam]\})^{1/p}$
does the job. This concludes the proof.
\end{proof}

\begin{remark}\label{normequ.rem}
\rm
It is worth noting that $\Ve\cdot\Ve_\lam$ satisfies the usual properties of a norm except for 
the triangle inequality. However, it follows from \eqref{normequ} that
\[
\Ve u+v \Ve_\lam \le c_2c_1^{-1}(\Ve u\Ve_\lam+\Ve v\Ve_\lam).
\]
Hence $\Ve\cdot\Ve_\lam$ is a norm or a quasinorm, depending on whether 
$c_2c_1^{-1}$ is smaller or larger than $1$. In fact it can be seen that 
$c_2c_1^{-1}>1$ if $\lambda>0$, so $\Ve\cdot\Ve_\lam$ is only a quasinorm in this case.
\end{remark}


\section{Existence of weak solutions}\label{existence.sec}

We will prove the existence of at least one positive, one negative, and one 
sign-changing solution of \eqref{problem} by constraint minimization of the
functional $S_\lam$ defined in \eqref{functional}. 
Since it is easier to obtain solutions of a given sign, 
we will focus our attention on the existence of a sign-changing solution,
and we will explain in the course of the proof how to modify it in order to get 
positive/negative solutions. The existence of a sign-changing solution is obtained
by minimizing $S_\lam$ on the `nodal Nehari set', which will be defined below. 

We define the positive and negative parts $u^\pm$ of a function
$u:\rn\to\real$ by $u^{\pm}(x):=\pm\max\{\pm u(x),0\}, \ x\in\rn$, so that
$u=u^++u^-$, with $\pm u^\pm \geqs0$.
It follows from \cite{sav} that $u^\pm\in W_A(\rn)$ whenever $u\in W_A(\rn)$.

\begin{theorem}\label{existence.thm}
Suppose that the hypotheses $(A1)$, $(A,\lam)$ and $(g1)$--\,$(g3)$ are satisfied. Then there 
exist $u_1,u_2,u_3\in W_A(\rn)$, with $u_1>0$ a.e., $u_2<0$ a.e., and $u_3^\pm\not\equiv 0$ a.e., 
such that $S_\lam'(u_i)=0, \ i=1,2,3$. 
\end{theorem}

\begin{remark}\label{poincare.rem}
\rm
When $\lam<0$, the condition $(A,\lam)$ of Theorem~\ref{existence.thm} is trivially 
satisfied. Then $A$ needs only satisfy assumption $(A1)$ and the conclusion
of Theorem~\ref{existence.thm} holds. The choice $A\equiv 1$ is allowed
in this case, yielding solutions in $W^{1,p}(\rn)$.
Observe that $\lam_A=0$ for $A\equiv 1$, reflecting the 
absence of Poincar\'e inequality on $\rn$.
In order to apply Theorem~\ref{existence.thm} with $\lam>0$, we need conditions on $A$
such that $\lam_A>0$.  
\end{remark}

\begin{proposition}\label{poincare.prop}
If $A$ satisfies $(A1)$ and $(A2)$
then $\lam_A>0$. Moreover, there exists $u^*\in W_A(\rn)$ such that
$\lam_A=\frac{\intrn|\nabla u^*|^p\diff x}{\intrn A(x)|u^*|^p\diff x}$.
\end{proposition}

\begin{proof}
By H\"older's inequality and Remark~\ref{classicsobolev.rem}, there exists $C>0$ such that
\[
\intrn A(x)|u|^p\diff x\le C\Ve A\Ve_{N/p}\Ve \nabla u\Ve_p^p, \quad u\in W_A(\rn).
\]
Furthermore, variational arguments similar to the proof of \cite[Theorem~1]{AlHu}
show that, when $(A1)$ and $(A2)$ hold, the infimum in \eqref{lambdaA}
is actually achieved.
\end{proof}

We will now prove Theorem~\ref{existence.thm}.
Without further mention,
we shall suppose that the hypotheses of Theorem~\ref{existence.thm} hold 
throughout the rest of this section.
Before we proceed with the bulk of the proof, 
let us first derive some elementary consequences of hypotheses $(g1)$--$(g3)$.

\begin{lemma}\label{consequences}
The function $g$ has the following properties.
\begin{itemize}
\item[(i)] $|g(x,s)|\le B(x)|s|^{q-1}$ and $|G(x,s)|\le \frac1q B(x)|s|^q$, for a.e. $x\in\rn$, 
and all $s\in\real$.
In particular, 
\[
g(x,s)=o(|s|^{p-1}) \quad \text{as} \ s\to 0, \ \text{uniformly for a.e.} \ x\in\R^N.
\]
\item[(ii)] $\disp\lim_{|s|\to\infty}\frac{|g(x,s)|}{|s|^{p-1}}=\infty$\quad and 
\quad$\disp\lim_{|s|\to\infty}\frac{|G(x,s)|}{|s|^{p}}=\infty$,\quad a.e. $x\in\rn$.
\item[(iii)] Letting $h(x,s)=\frac1p g(x,s)s - G(x,s)$, we have that 
\[
sG_s(x,s)>0 \quad\text{and}\quad sh_s(x,s)>0, \quad \text{a.e.} \  x\in\rn, \ s\neq0. 
\]
\end{itemize}
\end{lemma}

\begin{proof}
(i) follows immediately from $(g1)$. 
To prove (ii), first note that, for $s>R$,
\[
\frac{G'(x,s)}{G(x,s)}\geqs\frac{\theta}{s} \implies 
G(x,s)\geqs R^\theta G(x,R) s^\theta=R^\theta G(x,R)|s|^\theta,
\]
whereas, for $s<-R$,
\begin{align*}
\frac{G'(x,s)}{G(x,s)}\leqs\frac{\theta}{s}\ &\implies 
\int_s^{-R}\frac{G'(x,t)}{G(x,t)}\diff t\leqs \theta\int_s^{-R}\frac{\dif t}{t}
\implies \ln\frac{G(x, -R)}{G(x,s)} \leqs \ln\Big(\frac{-s}{R}\Big)^{-\theta}\\
& \implies G(x,s)\geqs R^{-\theta}G(x,-R)(-s)^{\theta}=R^{-\theta}G(x,-R)|s|^{\theta}.
\end{align*}
Therefore, $G(x,s)\geqs C(x)|s|^\theta$ for almost all $x\in\R^N$ and for all $|s|\geqs R$,
where $C(x):= \min\{R^\theta G(x,R),R^{-\theta}G(x,-R)\}$. Then $(g2)$ implies
that 
\[
|g(x,s)|\geqs \theta C(x)|s|^{\theta-1}, \quad\text{a.e.} \ x\in\R^N, \ |s|\geqs R,
\] 
with $C(x):= \min\{R^\theta G(x,R),R^{-\theta}G(x,-R)\}$, from which the limits in (ii) follow.
Finally, (iii) follows from $(g2)$ and $(g3)$.
\end{proof}

Let us now describe the variational setting we shall use to obtain critical points of the functional
$S_\lam$ defined in \eqref{functional}. By Lemma~\ref{derivative.lem}, 
$S_\lam \in C^1(W_A,\real)$
and, recalling the definition of $\Ve\cdot\Ve_\lam$ in \eqref{normlambda}, 
for all $\lam<\lam_A$ we define
\begin{align*}
J_{\lambda}(u) &= \langle S'_{\lambda}(u),u \rangle =
\Ve u\Ve_\lam^p- \int_{\R^N}g(x,u)u \diff x,\\
N_{\lambda} &=\bigl\{ u\in W_A\setminus\{0\}:J_{\lambda}(u)=0 \bigr\},\\
M_{\lambda} &=\bigl\{ u\in W_A:u^{\pm}\in N_{\lambda} \bigr\}\subset N_{\lambda}.
\end{align*}
The sets $N_\lam$ and $M_\lam$ are respectively known as the {\em Nehari manifold} 
and the {\em nodal Nehari set}. 
Clearly, $N_\lam$ contains all non-trivial solutions of \eqref{problem} while
$M_\lam$ contains all sign-changing solutions of \eqref{problem}.

For $u\in N_{\lambda}$,
it follows from Lemma~\ref{consequences}~(i), Proposition~\ref{sobolev.prop}
and Lemma~\ref{normequ.lem} that
\[
\|u\|_{\lambda}^p=\int_{\R^N}g(x,u)u\diff x \leqs C\int_{\R^N}B(x)|u|^q\diff x \leqs C\|u\|_{W_A}^q
\leqs C_{\lambda}\|u\|_{\lambda}^q.
\] 
Hence, letting $\delta_{\lambda}:=C_{\lambda}^{-1/(q-p)}$, we have
\begin{equation}\label{Nbzero}
\|u\|_{\lambda}\geqs\delta_{\lambda}>0, \quad u\in N_{\lambda}.
\end{equation}

Observing that
\[
J_{\lambda}(u)=0 \iff \|u\|^p_{\lambda}=\int_{\R^N}g(x,u)u\diff x,
\] 
it follows from Lemma~\ref{consequences}~(iii) that
\begin{equation}\label{bound}
S_{\lambda}(u)=\int_{\R^N}\frac1p g(x,u)u - G(x,u) \diff x
=\int_{\R^N}h(x,u) \diff x>0, \quad u\in N_{\lambda}.
\end{equation}
Therefore,\footnote{That $M_\lambda\neq\emptyset$ is easily seen from step 2 in the proof
of Proposition~\ref{min.lem}.}
\begin{equation}\label{min}
m_{\lambda}:=\inf_{M_{\lambda}}S_{\lambda}\geqs\inf_{N_{\lambda}}S_{\lambda}\geqs 0.
\end{equation}

We will now show that the Nehari manifold is diffeomorphic to the unit sphere in 
$(W_A,\Ve\cdot\Ve_\lam)$. Firstly,
similar arguments to the proof of Lemma~\ref{derivative.lem} show that 
$J_\lam\in C^1(W_A,\real)$,
and it follows from $(g3)$ that $\la J'_\lam(u),u\ra<0$ for all $u\in N_\lam$. Therefore,
by the submersion theorem, $N_\lam$ is a $C^1$ manifold of codimension 1 in $W_A$,
such that the tangent space $T_u N_\lam$ is transversal to $\real_+ u$, for all $u\in N_\lam$.

\begin{lemma}\label{proj.lem}
For any fixed $u\in W_A\setminus\{0\}$, there exists a unique $t=t_{\lambda}(u)>0$ such that 
$t_\lam(u)u\in N_\lam$. Furthermore the map $u\mapsto t_{\lambda}(u)u$ is a $C^1$ diffeomorphism
from $\{u\in W_A: \Ve u\Ve_\lam=1\}$ onto $N_\lam$, with inverse $u\mapsto u/\Ve u\Ve_\lam$.
Moreover, for any $u\in W_A\setminus\{0\}$, 
we have 
\begin{equation}\label{location}
t_\lam(u)<1 \  \text{if} \  J_\lam(u)<0 \quad\text{and}\quad\, t_\lam(u)>1 \  \text{if} \  J_\lam(u)>0.
\end{equation}
Finally, $S_{\lambda}(tu)$ is increasing for $t\in(0,t_\lam(u))$ and decreasing for 
$t\in(t_\lam(u),\infty)$,
with
\begin{equation}\label{max}
S_{\lambda}(t_{\lambda}(u)u)=\max_{t>0}S_{\lambda}(tu), \quad\text{for all} \ 
u\in W_A\setminus\{0\}.
\end{equation}
\end{lemma}

\begin{proof}
Define a $C^1$ function $\varphi:(0,\infty)\to\real$ by
 \[
 \varphi_u(t)=\frac{1}{t^p}J_{\lambda}(tu)=\|u\|_{\lambda}^p 
 -  \int_{\R^N}\frac{g(x,tu)}{t^{p-1}}u \diff x.
 \]
It follows from Lemma~\ref{consequences}~(i) and (ii) that
\[
\lim_{t\to 0^+}\varphi_u(t) = \|u\|_{\lambda}^p>0, \quad
\lim_{t\to +\infty}\varphi_u(t) = -\infty.
\]
Furthermore, by $(g3)$, $t\mapsto\varphi_u(t)$ is strictly decreasing on $(0,\infty)$,
from which the existence and uniqueness of $t_\lam(u)$ follow. 
The diffeomorphism statement is a consequence
of the implicit function theorem and the transversality of
$T_u N_\lam$ and $\real_+ u$.
\eqref{location} follows easily from the properties of $\varphi_u$, while
the behaviour of $S_\lam(tu), \ t>0$, follows from the calculation
\begin{align*}
\frac{\dif}{\dif t}S_{\lambda}(tu) &=\langle S'_{\lambda}(tu), u \rangle\\
 &=\int_{\R^N}t^{p-1}(|\nabla u|^p - \lambda A(x)|u|^p)\diff x - \int_{\R^N}g(x,tu)u\diff x\\
 &=t^{-1}\Bigl( \|tu\|_{\lambda}^p - \int_{\R^N}g(x,tu)tu\diff x\Bigr)=t^{-1}J_{\lambda}(tu).
\end{align*}
The lemma is proved.
\end{proof}

\begin{proposition}\label{min.lem}
The infimum $m_{\lambda}$ defined in \eqref{min} is achieved.
\end{proposition}

\begin{proof}
In the course of this proof, we will take the liberty of passing to subsequences
when necessary, without mentioning it explicitly. The proof proceeds in two steps.

\smallskip
\noindent
{\em 1.~Boundedness of a minimizing sequence.}
Consider $(u_n)\subset M_{\lambda}$ such that $S_{\lambda}(u_n)\to m_{\lambda}$,
and suppose by contradiction that $\Ve u_n\Ve_\lam\to\infty$ as $n\to\infty$.
Now let
\[
v_n=p[m_\lam+1]^{1/p}\frac{u_n}{\Ve u_n\Ve_\lam}. 
\]
Since the sequence $(v_n)$
is bounded, we can suppose that there exists $v\in W_A$ such that 
$v_n\wto v$ weakly in $W_A$. By Lemma~\ref{compcont.lem}, we have 
\[
S(v_n)\to m_\lam+1 - \intrn G(x,v) \diff x.
\]
On the other hand, since $u_n\in N_\lam$, it follows
from \eqref{max} that $S(v_n)\leqs S(u_n)$. We shall thus reach a contradiction by showing
that $v\equiv0$. If it is not the case, there exists a set $\Omega\subset\rn$ with 
positive measure, and a number $\delta>0$, such that $\essinf_\Omega |v|\geqs\delta$.
Invoking Proposition~\ref{sobolev.prop} and Egorov's theorem, we can suppose that
\[
\essinf_\Omega |v_n|\geqs\frac{\delta}{2}>0, \quad n\geqs n_0,
\]
for some large enough $n_0\in\nat$. Since $G(x,0)\equiv0$ and the supports of
$v_n^+$ and $v_n^-$ are disjoint, it follows from Lemma~\ref{consequences}~(iii) that
\begin{align*}
\intrn\frac{G(x,u_n)}{\Ve u_n\Ve_\lam^p}\diff x 
& 
\geqs 
\int_\Omega\frac{G\big(x,\frac{\Ve u_n\Ve_\lam}{p[m_\lam+1]^{1/p}}v_n\big)}
{\Ve u_n\Ve_\lam^p}\diff x \\
&
= \int_\Omega\frac{G\big(x,\frac{\Ve u_n\Ve_\lam}{p[m_\lam+1]^{1/p}}v_n^+\big)
+G\big(x,\frac{\Ve u_n\Ve_\lam}{p[m_\lam+1]^{1/p}}v_n^-\big)}
{\Ve u_n\Ve_\lam^p}\diff x \\
&
\geqs 
\int_\Omega\frac{G\big(x,\frac{\Ve u_n\Ve_\lam}{p[m_\lam+1]^{1/p}}\frac{\delta}{2}\big)
+G\big(x,\frac{\Ve u_n\Ve_\lam}{p[m_\lam+1]^{1/p}}(-\frac{\delta}{2})\big)}
{\Ve u_n\Ve_\lam^p}\diff x, \quad n\geqs n_0.
\end{align*}
Then Lemma~\ref{consequences}~(ii) yields
\[
\intrn\frac{G(x,u_n)}{\Ve u_n\Ve_\lam^p}\diff x\to\infty \quad\text{as} \  n\to\infty.
\]
However, on the other hand,
\[
\intrn\frac{G(x,u_n)}{\Ve u_n\Ve_\lam^p}\diff x = 
\frac{\frac1p\Ve u_n\Ve_\lam^p - S_\lam(u_n)}{\Ve u_n\Ve_\lam^p} 
\to \frac1p \quad\text{as} \  n\to\infty,
\]
which gives the desired contradiction. Therefore, any minimizing sequence $(u_n)$
is indeed bounded.

\smallskip
\noindent
{\em 2.~Existence of a minimizer.}
Let $(u_n)\subset M_{\lambda}$ such that $S_{\lambda}(u_n)\to m_{\lambda}$. 
Since $(u_n)$ is bounded in $W_A$, there exists $u\in W_A$ 
such that $u_n\wto u$ and $u_n^{\pm}\wto u^{\pm}$ weakly in $W_A$ as $n\to\infty$. 
It immediately follows from the weak lower semicontinuity of $u\mapsto \Vert\nabla u\Vert^p_p$,
and from Lemma~\ref{compcont.lem}, that
\[
S_\lam(u)\leqs \liminf_{n\to\infty} S_\lam(u_n)=m_\lam.
\]
Hence we need only prove that $u^{\pm}\in N_{\lambda}$.
We first observe that $u^{\pm}\not\equiv 0$.
Indeed, by Lemma~\ref{compcont.lem} and \eqref{Nbzero},
\begin{equation}\label{norms}
\int_{\R^N}g(x,u^{\pm})u^{\pm}\diff x = \lim_{n\to\infty}\int_{\R^N}g(x,u_n^{\pm})u_n^{\pm}\diff x 
= \lim_{n\to\infty}\|u_n^{\pm}\|_{\lambda}^p\geqs\delta_{\lambda}^p>0.
\end{equation}
Invoking again the weak lower semicontinuity of $u\mapsto \Vert\nabla u\Vert^p_p$
and Lemma~\ref{compcont.lem}, it follows from \eqref{norms} that
\[
J_\lam(u^\pm)=\Ve u^\pm\Ve_\lam^p - \lim_{n\to\infty}\|u_n^{\pm}\|_{\lambda}^p\leqs0.
\]
Suppose by contradiction that
\[
\|u^+\|_{\lambda}^p<\liminf_{n\to\infty}\|u_n^+\|_{\lambda}^p.
\]
Then $t^+:=t_{\lambda}(u^+)<1$ and $t^-:=t_{\lambda}(u^-)\leqs 1$,
$t^+u^++t^-u^-\in M_{\lambda}\subset N_{\lambda}$, and so
\[
S_{\lambda}(t^+u^++t^-u^-) = \int_{\R^N}h(x,t^+u^++t^-u^-)\diff x
\]
by \eqref{bound}.
Since $h(x,0)\equiv 0$ and the supports of $u^+$ and $u^-$ are disjoint, 
it follows from Lemma~\ref{consequences}~(iii) that
\begin{align*}
S_{\lambda}(t^+u^++t^-u^-)
&=\int_{\R^N} h(x,t^+u^++t^-u^-)\diff x \\
&= \int_{\R^N}h(x,t^+u^+)\diff x + \int_{\R^N}h(x,t^-u^-)\diff x\\
&<\int_{\R^N} h(x, u^+)\diff x + \int_{\R^N} h(x, u^-)\diff x = \int_{\R^N} h(x, u^++u^-)\diff x \\
&= \int_{\R^N} h(x,u)\diff x =\lim_{n\to\infty}\int_{\R^N} h(x,u_n)\diff x
=\lim_{n\to\infty}S_{\lambda}(u_n)=m_{\lambda}.
\end{align*}
This contradiction concludes the proof.
\end{proof}

\begin{remark}
\rm
If $\lam<0$ then $\Ve\cdot\Ve_\lam$ is a norm, $(W_A,\Ve\cdot\Ve_\lam)$ is 
uniformly convex, and the proof shows that $u_n\to u$ in $W_A$ (up to a subsequence).
\end{remark}

We are now in a position to complete the 

\medskip
\noindent{\it Proof of Theorem~\ref{existence.thm}.}
Proposition~\ref{min.lem} yields an element $u_3\in W_A$ that minimizes $S_\lam$ on the nodal
Nehari set $M_\lam$. To conclude the proof of Theorem~\ref{existence.thm}, we will now show
that $S_\lam'(u_3)=0$. The existence of the critical points $u_1$ and $u_2$ (with $u_1>0$ a.e. 
and $u_2<0$ a.e.) follows similarly, by minimizing $S_\lam$ over $N_\lam^\pm$ instead of $M_\lam$, 
where
\[
N_\lam^\pm=\{u\in N_\lam: u^\mp=0\}.
\] 
Since $M_\lam$ is not a submanifold of $W_A$, we cannot use the Lagrange multiplier
theorem to infer that the minimizer $u_3$ is indeed a critical point of $S_\lam$.
To overcome this difficulty, we appeal to a theorem of 
Miranda\footnote{which is essentially a version of Brouwer's fixed point theorem}, which was
first established in \cite{mir}. 
For the reader's convenience, we recall here the two-dimensional version of this result. An
elegant proof can be found in \cite{vra}.

\begin{lemma}\label{mir.lem}
Let $L>0$, $R=(-L,L)^2\subset\real^2$ and consider a continuous function 
$F=(F_1,F_2):\overline{R}\to\real^2$ which satisfies $F(t,s)\neq0$ for all $(s,t)\in\partial R$, and
the following conditions on the boundary $\partial R:$
\[
F_1(-L,t)\ge 0, \quad F_1(L,t)\le 0, \quad F_2(s,-L)\ge 0, \quad F_2(s,L)\le 0.
\]
In other words, the vector field $F$, evaluated on the boundary $\partial R$, 
always points towards the interior of $R$. Then there exists $(s_0,t_0)\in R$ such that
$F(s_0,t_0)=0$.
\end{lemma}

We apply this lemma in the following way. Suppose by contradiction that $S_\lam'(u_3)\neq0$.
Then there is $\ffi\in W_A$ such that $\la S_\lam'(u_3), \ffi \ra=-2$ and so, 
by continuity of $S_\lam'$, there is an $\eps>0$ such that
\begin{equation}\label{neg}
\la S_\lam'(tu_3^+ + su_3^-+r\ffi),\ffi \ra<-1
\end{equation}
for all $r\in(0,\eps]$ and all $(s,t)\in \overline{R}$, where $R=(1-\eps,1+\eps)^2\subset\real^2$. 
Now consider a continuous function $\eta:\overline{R}\to[0,\eps]$ such that 
$\eta(1,1)=\eps, \ \eta(\partial R)=0$, and $\eta\neq0$ on $R$. 
We define $F:\overline{R}\to\real^2$ by
\[
F(s,t)=\Big(J_\lam\big((tu_3^+ + su_3^-+\eta(t,s)\ffi)^-\big),
J_\lam\big((tu_3^+ + su_3^-+\eta(t,s)\ffi)^+\big)\Big).
\]
First of all, it is clear that $F$ is continuous. Next, for $s=1-\eps$ and
$t\in[1-\eps,1+\eps]$, 
using the function $\ffi_u(t)$ introduced in the proof of Lemma~\ref{proj.lem} we have
\[
\frac{F_1(1-\eps,t)}{(1-\eps)^p}=\frac{J_\lam((1-\eps)u_3^-) }{(1-\eps)^p}=
\ffi_{u_3^-}(1-\eps)>\ffi_{u_3^-}(1)=0
\]
since $u_3^-\in N_\lam$. 
Similar arguments show that $F_1(1+\eps,t)<0$ and $\mp F_2(s,1\pm\eps)>0$. Hence, the
hypotheses of Lemma~\ref{mir.lem} are satisfied and there exists $(s_0,t_0)\in R$ such that
$F(s_0,t_0)=0$. Remarking that $t_0u_3^+ + s_0u_3^-+\eta(t_0,s_0)\ffi\neq0$ 
by \eqref{neg}, it follows that $t_0u_3^+ + s_0u_3^-+\eta(t_0,s_0)\ffi\in M_\lam$.
We will reach a contradiction by showing that 
$S_\lam(t_0u_3^+ + s_0u_3^-+\eta(t_0,s_0)\ffi)<m_\lam$. By \eqref{neg} we have
\begin{align*}
S_\lam(t_0u_3^+ + s_0u_3^-+\eta(t_0,s_0)\ffi)
	&= S_\lam(t_0u_3^+ + s_0u_3^-) \\
	&\qquad\qquad	+ \int_0^{\eta(t_0,s_0)}
	\la S_\lam'(tu_3^+ + su_3^-+r\ffi),\ffi \ra \,\mathrm{d}r  \\
	&< S_\lam(t_0u_3^+ + s_0u_3^-)  - \eta(t_0,s_0).
\end{align*}
If $(s_0,t_0)=(1,1)$ then $S_\lam(t_0u_3^+ + s_0u_3^-+\eta(t_0,s_0)\ffi)<m_\lam-\eps$ 
and we are done. So suppose that $(s_0,t_0)\neq(1,1)$. 
Since $u_3^\pm\in N_\lam$ it follows from \eqref{max} that
\begin{align*}
S_\lam(t_0u_3^+ + s_0u_3^-+\eta(t_0,s_0)\ffi)
	&< S_\lam(t_0u_3^+ + s_0u_3^-)  - \eta(t_0,s_0)\\
	&=S_\lam(t_0u_3^+)+S_\lam(s_0u_3^-) - \eta(t_0,s_0)\\
	&\le S_\lam(u_3^+)+S_\lam(u_3^-) - \eta(t_0,s_0)\\
	&=S_\lam(u_3)- \eta(t_0,s_0)<m_\lam,
\end{align*}
yielding the desired contradiction. This completes the proof of Theorem~\ref{existence.thm}.
\hfill$\Box$


\section{Regularity and conclusion of the proof}\label{reg.sec}
We will now prove a regularity result for the solutions given by Theorem~\ref{existence.thm}. 
We say that $u \in C^{1,\alpha}_\mathrm{loc}(\rn)$ if
for any compact $K\subset\rn$ there exists $\alpha=\alpha(K)$ such that $u \in C^{1,\alpha}(K)$.
Once $C^{1,\alpha}$ regularity is proved, we will prove the remaining statements of 
Theorem~\ref{main.thm}, about the nodal properties of the solutions.

\begin{proposition}
Suppose that the hypotheses $(A1)$, $(A2)$ and $(g1)$ hold.
Then any weak solution $u$ of \eqref{problem} satisfies $u\in C^{1,\alpha}_\mathrm{loc}(\rn)$. 
\end{proposition}

\begin{proof}
Assumption $(A1)$ implies that 
$u\in W^{1,p}_\mathrm{loc}(\rn)\subset L^{p^*}_\mathrm{loc}(\rn)$.
It then follows from \cite{tolk}
that $u\in C^{1,\alpha}_\mathrm{loc}(\rn)$, provided we know {\em a priori} that
$u\in L^\infty_\mathrm{loc}(\rn)$. 
For a fixed compact $K\subset\rn$,
we briefly explain how the results of \cite{lad} imply that 
$u\in L^\infty(K)$. Firstly, the integrability conditions on $A$ and $B$ given in hypotheses
$(A1)$, $(A2)$ and $(g1)$ precisely ensure that the assumptions (7.1) and (7.2) 
in \cite[Chap.~4,~Sec.~7]{lad} (together with conditions 1)--3) there) hold. 
Thus the idea is to apply Theorem~7.1 of \cite[Chap.~4,~Sec.~7]{lad}. In the present context,
the parameters $m$ and $q$ appearing there are given by $m=p$ and $q=p^*$.
We need only explain here why the conclusion of this theorem still holds without 
the hypothesis that
\begin{equation}\label{boundary}
\esssup_{\partial K} |u|<\infty. 
\end{equation}
This assumption is used in the proof of 
\cite[Theorem~7.1,~Chap.~4]{lad} in the two following instances. First, to derive
the estimate (7.3), the test function $\eta(x)=\max\{u(x)-k,0\}$ is used, where $k$ is a positive
parameter such that $k\ge \esssup_{\partial K} |u|$. This restriction is due to the definition
of a weak solution in \cite[Chap.~4]{lad}, requiring that $\eta\in W_0^{1,p}(K)$. Our definition
of a weak solution (see \eqref{weak}) allows us to merely consider $\eta\in W^{1,p}(K)$, and we
can derive the estimate (7.3) in the same way as in \cite{lad}. Finally, assumption \eqref{boundary}
is used to conclude the proof of Theorem~7.1 by
invoking Theorem~5.1 of \cite[Chap.~2]{lad}. It turns out that Theorem~5.2 of \cite[Chap.~2]{lad}
does the job as well, and does not require \eqref{boundary}. This completes the proof.
\end{proof}

We can now finish the 

\medskip
\noindent{\it Proof of Theorem~\ref{main.thm}.}
It only remains to show that $u_1>0$, $u_2<0$, and $u_3$ has exactly two nodal domains,
as defined in Section~\ref{tuning.sec}. The positivity of $u_1$ and the negativity of $u_2$
follow from the strong maximum principle, see e.g. Theorem~5 in \cite{va}. Regarding $u_3$,
we already know from the previous results that it has two nodal domains.
Suppose by contradiction that $u_3$ has (at least) 
three distinct nodal domains $\Omega_i\subset\rn, \ i=1,2,3$, and define
\[
v_i(x)=
\begin{cases}
u_3(x) &\text{if} \  x\in \Omega_i,\\
0 &\text{otherwise}.
\end{cases}
\]
Clearly $v_i\in W_A, \ i=1,2,3$, and
without loss of generality we can suppose that $v_1>0$ and $v_2<0$. 
Since $S_\lam'(u_3)=0$, it follows that $v_i\in N_\lam, \ i=1,2,3$,
$v_1+v_2\in M_\lam$, and so \eqref{bound} implies that
\[
m_\lam\le S_\lam(v_1+v_2)< S_\lam(u_3)=m_\lam.
\]
This contradiction concludes the proof.
\hfill$\Box$


\appendix

\section{Compactness}\label{compact.sec}

Since we were not able to find the exact result we need in the literature,
we now give a proof of the compactness of the embedding 
$W_A(\Omega) \subset L_B^r(\Omega)$ in Proposition~\ref{sobolev.prop}. We shall
make extensive use of the classic H\"older and interpolation inequalities, which hold
in the weighted spaces $L_B^q(\Omega)$, as in the usual case $B\equiv 1$.

\medskip
\noindent
{\em Proof of Proposition~\ref{sobolev.prop} (continued).}
We start by assuming that $\Omega\subset\rn$ is an open bounded domain with $C^1$ boundary.
We will first explain how the proof of the classic Rellich-Kondrachov theorem can be adapted to the
the present case. We follow the
proof of Brezis \cite[Th\'eor\`eme~IX.16]{Br}, which is based on a criterion of strong compactness
in $L^p$ spaces \cite[Corollaire~IV.26]{Br}. Note that \cite[Corollaire~IV.26]{Br} is
a consequence of the famous Riesz-Fr\'echet-Kolmogorov theorem \cite[Th\'eor\`eme~IX.25]{Br}. 
It is easy to see that the proof of both \cite[Th\'eor\`eme~IX.25]{Br} and 
\cite[Corollaire~IV.26]{Br}
remain virtually unchanged in the case of a weighted $L^p$ space such as $L_B^r(\Omega)$. 
Therefore, we can merely follow the proof of \cite[Th\'eor\`eme~IX.16]{Br} in the case $p<N$. 

Letting $\calF$ be the unit ball in $W_A(\Omega)$, this amounts to verifying assumptions 
(IV.23) and (IV.24) of \cite[Corollaire~IV.26]{Br}. Following Brezis,
for an open set $\omega\subset\Omega$ such that $\overline\omega$
is compact and $\overline\omega\subset\Omega$, we write $\omega\subset\subset\Omega$.
For a given $h\in\rn$, we also define the translate $\tau_h u$ of a function $u$ by 
$\tau_h u(x)=u(x+h), \ x\in\rn$. Assumptions (IV.23) and (IV.24) of \cite[Corollaire~IV.26]{Br} 
now read
\begin{align}\label{23}
&\forall\,\eps>0 \quad \forall\,\omega\subset\subset\Omega, 
\quad \exists\,\delta\in(0,\dist(\omega,\Omega^\mathrm{c})) \quad\text{s.t.} \notag \\
&\Ve \tau_h u - u \Ve_{L^q_B(\omega)}<\eps \quad \forall\,h\in\rn \ \text{s.t.} \ |h|<\de, 
\quad \forall\,u\in \calF
\end{align}
and
\begin{equation}\label{24}
\forall\,\eps>0 \quad \exists\,\omega\subset\subset\Omega \quad\text{s.t.}\quad
\Ve u \Ve_{L^q_B(\Omega\setminus\omega)}<\eps \quad \forall\,u\in \calF.
\end{equation}
To prove \eqref{23}, consider $\al\in(0,1]$ such that $\frac1q=\frac{\al}{1}+\frac{1-\al}{p^*}$. 
By interpolation,
\begin{equation}\label{interpolation}
\Ve \tau_h u - u \Ve_{L^q_B(\omega)} \le 
\Ve \tau_h u - u \Ve_{L^1_B(\omega)}^\al \Ve \tau_h u - u \Ve_{L^{p^*}_B(\omega)}^{1-\al}.
\end{equation}
Now following the proof of \cite[Proposition~IX.3]{Br}, it is easily seen that
\begin{equation}\label{L1est}
\Ve \tau_h u - u \Ve_{L^1_B(\omega)}\le \Ve B\Ve_\infty\Ve\nabla u\Ve_{L^1(\Omega)}|h|.
\end{equation}
Hence, recalling that $\Omega$ is bounded and using \eqref{xtrmsobolev}, we have
\begin{align*}
\Ve \tau_h u - u \Ve_{L^q_B(\omega)}
&\le \Ve B\Ve_\infty^\al \Ve\nabla u\Ve_{L^1(\Omega)}^\al |h|^\al 
(2 \Ve  u \Ve_{L^{p^*}_B(\Omega)})^{1-\al}\\
&\le 2^{1-\al} \Ve B\Ve_\infty \Ve\nabla u\Ve_{L^1(\Omega)}^\al 
\Ve  u \Ve_{L^{p^*}(\Omega)}^{1-\al}|h|^\al\\
&\le C \Ve\nabla u\Ve_{L^p(\Omega)}^\al \Ve \nabla u \Ve_{L^p(\Omega)}^{1-\al}|h|^\al\\
&\le C|h|^\al \quad (\text{since}  \  \Ve\nabla u\Ve_{L^p(\Omega)}\le 1 \ \text{for} \ u\in\calF),
\end{align*}
which proves \eqref{23}. On the other hand, for all $u\in\calF$, it follows from 
H\"older's inequality and \eqref{xtrmsobolev} that
\begin{align*}
\Ve u\Ve_{L^q_B(\Omega\setminus\omega)}^q
&=\int_{\Omega\setminus\omega}|u|^qB(x)\diff x
	\le \Ve B\Ve_\infty|\Omega\setminus\omega|^\frac{p^*-q}{p^*}
			\Ve u\Ve^q_{L_B^{p^*}(\Omega)}\\
&\le C |\Omega\setminus\omega|^\frac{p^*-q}{p^*}\Ve u\Ve^q_{L^{p^*}(\Omega)} 
\le C \Ve\nabla u\Ve_{L^p(\Omega)}^q |\Omega\setminus\omega|^\frac{p^*-q}{p^*} \\
&\le C |\Omega\setminus\omega|^\frac{p^*-q}{p^*}.
\end{align*}
Therefore, 
$\Ve u \Ve_{L^q_B(\Omega\setminus\omega)}\le 
C |\Omega\setminus\omega|^{\frac{1}{q}-\frac{1}{p^*}}$,
which proves \eqref{24} and concludes the proof that the embedding is compact when 
$\Omega$ is a smooth bounded domain.

We now consider the case $\Omega=\rn$. Let $(u_n)\subset W_A(\rn)$ be a bounded 
sequence. We will show that $(u_n)$ is relatively compact in $L_B^q(\rn)$. 
Firstly, since $W_A(\rn)$ is 
reflexive, we can suppose that $u_n\wto u$ weakly in $W_A(\rn)$, for some $u\in W_A(\rn)$. 
Also, denoting by $B(0,R)$ the ball 
of radius $R$ centred at $x=0$ in $\rn$, $u_n\ve_{B(0,R)}\wto u\ve_{B(0,R)}$ weakly in 
$W_A(B(0,R))$, and we already know that (up to a subsequence) $u_n\ve_{B(0,R)}\to u\ve_{B(0,R)}$ 
in $L_B^q(B(0,R))$. 
Furthermore, by H\"older's inequality and \eqref{xtrmsobolev},
\begin{align}\label{exterior}
\int_{|x|\ge R} |u_n-u|^q B(x) \diff x 
&\le C \int_{|x|\ge R} (|u_n|^q+|u|^q) B(x) \diff x \notag \\
&\le C \Big(\int_{|x|\ge R}(|u_n|^q+|u|^q)^{p^*/q}\diff x\Big)^{q/p^*} 
\Ve B\Ve_{L^\frac{p^*}{p^*-q}(|x|\ge R)} \notag\\
&\le C\big(\Ve u_n\Ve_{p^*}^q+\Ve u\Ve_{p^*}^q\big)\Ve B\Ve_{L^\frac{p^*}{p^*-q}(|x|\ge R)}\notag\\
&\le C \big(\Ve \nabla u_n\Ve_{p}^q+\Ve\nabla u\Ve_{p}^q\big)
\Ve B\Ve_{L^\frac{p^*}{p^*-q}(|x|\ge R)}\notag\\
&\le C\Ve B\Ve_{L^\frac{p^*}{p^*-q}(|x|\ge R)}.
\end{align}
Since $B\in L^\frac{p^*}{p^*-q}(\rn)$, the right-hand side of \eqref{exterior} can be made 
arbitrarily small by choosing $R>0$ large enough, uniformly in $n$, which concludes the proof.
\hfill$\Box$

\begin{remark}\label{localintegr.rem}
\rm
Let us now explain how the hypothesis $B\in L^\infty(\rn)$ in $(g1)$ can be relaxed to
a local integrability condition in the above proof, and hence throughout the whole paper.
Choosing $t\in(q,p^*)$ and $\al\in(0,1]$ such that $\frac1q=\frac{\al}{1}+\frac{1-\al}{t}$, 
we start by replacing the interpolation inequality \eqref{interpolation} by
\begin{equation}\label{interpolation'}
\Ve \tau_h u - u \Ve_{L^q_B(\omega)} \le 
\Ve \tau_h u - u \Ve_{L^1_B(\omega)}^\al \Ve \tau_h u - u \Ve_{L^t_B(\omega)}^{1-\al}.
\end{equation}
Then, instead of \eqref{L1est}, the proof of \cite[Proposition~IX.3]{Br} can be modified to show 
that
\begin{equation}\label{L1est'}
\Ve \tau_h u - u \Ve_{L^1_B(\omega)}\le \Ve B\Ve_{L^\frac{p}{p-1}(\Omega)}
\Ve\nabla u\Ve_{L^p(\Omega)}|h|.
\end{equation}
On the other hand,
\begin{align}\label{est'}
\Ve \tau_h u - u \Ve_{L^t_B(\omega)}
&\le 2 \Ve  u \Ve_{L^t_B(\Omega)}
	\le 2 \Ve B\Ve^{\frac1t}_{L^\frac{p^*}{p^*-t}(\Omega)}\Ve u\Ve_{L^{p^*}(\Omega)}\notag\\
&\le C \Ve B\Ve^{\frac1t}_{L^\frac{p^*}{p^*-t}(\Omega)} \Ve\nabla u\Ve_{L^p(\Omega)}.
\end{align}
That \eqref{23} holds now follows from \eqref{interpolation'}, \eqref{L1est'} and \eqref{est'}, 
provided\footnote{Note that $\frac{p}{p-1}>\frac{p^*}{p^*-t} \iff t<\frac{N}{N-p}$.}
\begin{equation}\label{localintegr}
B \in L^s_\mathrm{loc}(\rn), \quad\text{where}\quad s=\max\Big\{\frac{p}{p-1},
\frac{p^*}{p^*-t}\Big\}.
\end{equation}
Furthermore, we have
\begin{align*}
\Ve u\Ve_{L^q_B(\Omega\setminus\omega)}^q
&=\int_{\Omega\setminus\omega}|u|^qB(x)\diff x
	\le \Ve B\Ve_{L^\frac{p^*}{p^*-q}(\Omega\setminus\omega)}\Ve u\Ve^q_{L^{p^*}(\Omega)}\\
&\le C \Ve B\Ve_{L^\frac{p^*}{p^*-q}(\Omega\setminus\omega)}\Ve \nabla u\Ve^q_{L^p(\Omega)},
\end{align*}
and so \eqref{24} follows from the assumption that $B\in L^{\frac{p^*}{p^*-q}}(\R^N)$. 
This completes
the proof of compactness of the embedding $W_A(\Omega) \subset L_B^q(\Omega)$ when $\Omega$
is a bounded domain with smooth boundary. The case $\Omega=\rn$ then follows as before. 
Hence, we see that Proposition~\ref{sobolev.prop} holds provided 
$B\in L^{\frac{p^*}{p^*-q}}(\R^N)$ and $B$ satisfies \eqref{localintegr} for some $t\in(q,p^*)$.
\end{remark}


\section{Differentiability}\label{derivative.sec}

This appendix is devoted to the proof of Lemma~\ref{derivative.lem}. Before we proceed with the
proof, let us first remark that, thanks to Proposition~\ref{sobolev.prop},
\[
u\in W_A(\rn) \implies u \in L_A^p(\rn) \  \text{and}  \  u \in L_B^q(\rn),
\]
where $B$ and $q$ have been introduced in hypothesis $(g1)$.
As in Appendix~\ref{compact.sec}, we will again take advantage of the H\"older inequality in 
the weighted Lebesgue spaces $L_A^p(\rn)$ and $L_B^q(\rn)$.

\medskip
\noindent
{\em Proof of Lemma~\ref{derivative.lem}.}
For given $u,v \in W_A(\rn)$ we start by computing the G\^ateaux derivative $DS_\lam(u)$ of 
$S_\lam$ in the direction $v$, and we show that it is equal to the right side of 
\eqref{derivative}. That is, we
compute $\lim_{t\to0} \frac1t[S_\lam(u+tv)-S_\lam(u)]$. It follows from H\"older's inequality that
$|\nabla u|^{p-2}\nabla u \cdot \nabla v \in L^1(\rn)$. The derivation of the first term 
then follows in a standard manner, using the mean-value theorem and the dominated convergence
theorem. Let us now consider the other two terms in more details. For $t\neq0$, and $x\in\rn$, it 
follows from the mean-value theorem that there exists $s=s(t,x)\in[0,1]$ such that
\[
|u+tv|^p-|u|^p=p|u+stv|^{p-2}(u+stv)tv,
\]
and so
\[
\frac1t\frac1p A(x) (|u+tv|^p-|u|^p) \to A(x) |u|^{p-2}uv \quad \text{as} \ t\to0, \ 
\text{for a.e.} \ x\in\rn.
\]
Moreover,  
\[
\Big|\frac1t\frac1p A(x) (|u+tv|^p-|u|^p)\Big| \le A(x) |u+stv|^{p-1}|v| 
\le C A(x) (|u|^{p-1}|v|+|v|^p),
\]
where $A|v|^p \in L^1(\rn)$ since $v\in W_A$, and $A(x)|u|^{p-1}|v| \in L^1(\rn)$ by
H\"older's inequality in $L_A^p(\rn)$. Hence, by dominated convergence,
\[
\frac1t\frac1p\intrn A(x) (|u+tv|^p-|u|^p)\diff x \to \intrn A(x) |u|^{p-2}uv\diff x \quad \text{as} \ t\to0.
\]

To deal with the last term, we apply again the mean-value theorem, which yields a number 
$s=s(t,x)\in[0,1]$ such that
\[
\frac1t (G(x,u+tv)-G(x,u))=\frac1t g(x,u+stv)tv \to g(x,u)v  \quad \text{as} \ t\to0, \ 
\text{for a.e.} \ x\in\rn.
\]
Also, by $(g1)$,
\[
|g(x,u+stv)v| \le B(x)|u+stv|^{q-1}|v| \le C B(x) (|u|^{q-1}|v|+|v|^q) \in L^1(\rn),
\]
thanks to Proposition~\ref{sobolev.prop} with $r=q$. It then follows by dominated convergence that
\[
\frac1t\intrn (G(x,u+tv)-G(x,u)) \diff x \to \intrn g(x,u)v \diff x  \quad \text{as} \ t\to0.
\]
We have thus proved that the G\^ateaux derivative $DS_\lam(u)v$ exists and is equal to the
right-hand side of \eqref{derivative}

To complete the proof, we will now show that $DS_\lam(u)\in W_A^*$ for all $u\in W_A$, and that the
mapping $u\mapsto DS_\lam(u)$ is continuous. H\"older's inequality yields
\begin{align*}
\Big| \intrn |\nabla u|^{p-2}\nabla u \cdot \nabla v \diff x \Big| 
&\le \intrn |\nabla u|^{p-1}|\nabla v| \diff x \le \Ve |\nabla u|^{p-1} \Ve_{\frac{p}{p-1}} 
\Ve \nabla v \Ve_{p}\\
&= \Ve \nabla u \Ve_{p}^{p-1} \Ve \nabla v \Ve_{p} 
\le \Ve u\Ve_{W_A}^{p-1} \Ve v \Ve_{W_A},
\end{align*}
\begin{align*}
\Big|\intrn A(x) |u|^{p-2}uv \diff x\Big| 
& \le \intrn |u|^{p-1}|v| A(x) \diff x \le \Ve |u|^{p-1} \Ve_{L_A^\frac{p}{p-1}}\Ve v\Ve_{L_A^p}\\
&= \Ve u \Ve_{L_A^p}^{p-1} \Ve v\Ve_{L_A^p} \le \Ve u\Ve_{W_A}^{p-1} \Ve v \Ve_{W_A},
\end{align*}
and
\begin{align*}
\Big|\intrn g(x,u)v \diff x\Big| 
& \le \intrn |g(x,u)||v| \diff x \le \intrn |u|^{q-1}|v| B(x) \diff x \\
& \le \Ve |u|^{q-1} \Ve_{L_B^\frac{q}{q-1}}\Ve v\Ve_{L_B^q}\\ 
&=  \Ve u \Ve_{L_B^q}^{q-1} \Ve v\Ve_{L_B^q} \le C \Ve u\Ve_{W_A}^{q-1} \Ve v \Ve_{W_A},
\end{align*}
where the last inequality follows from Proposition~\ref{sobolev.prop}.
These estimates show that $DS_\lam(u)\in W_A^*$, for all $u\in W_A$.

To prove that $u\mapsto DS_\lam(u)$ is continuous, consider $(u_n)\subset W_A$ such that
$u_n\to u$ in $W_A$. We will show that
\begin{equation}\label{limito}
\Ve DS_\lam(u_n)-DS_\lam(u) \Ve 
= \sup_{v \in W_A\setminus\{0\}} \frac{|\la DS_\lam(u_n)-DS_\lam(u),v \ra|}{\Ve v\Ve_{W_A}} \to 0
\quad\text{as} \ n\to\infty. 
\end{equation}
We have
\begin{align}
|\la DS_\lam(u_n)-DS_\lam(u),v \ra| 
&\le \intrn ||\nabla u_n|^{p-2}\nabla u_n -|\nabla u|^{p-2}\nabla u| |\nabla v|\diff x\notag \\
&\phantom{\le} + |\lam| \intrn A(x)||u_n|^{p-2}u_n-|u|^{p-2}u||v|\diff x\notag\\
&\phantom{\le} + \intrn |g(x,u_n)-g(x,u)||v|\diff x.\notag
\end{align}
Using H\"older's inequality in the same fashion as above, we get
\begin{multline}\label{est1}
\intrn ||\nabla u_n|^{p-2}\nabla u_n -|\nabla u|^{p-2}\nabla u| |\nabla v|\diff x\\
\le \Ve |\nabla u_n|^{p-2}\nabla u_n -|\nabla u|^{p-2}\nabla u \Ve_{\frac{p}{p-1}} \Ve v\Ve_{W_A},
\end{multline}
\begin{equation}\label{est2}
\intrn A(x)||u_n|^{p-2}u_n-|u|^{p-2}u||v|\diff x
\le \Ve |u_n|^{p-2}u_n-|u|^{p-2}u \Ve_{L_A^\frac{p}{p-1}}\Ve v\Ve_{L_A^p}
\end{equation}
and
\begin{equation}\label{est3}
\intrn |g(x,u_n)-g(x,u)||v|\diff x
\le C \Ve |u_n|^{q-2}u_n-|u|^{q-2}u \Ve_{L_B^\frac{q}{q-1}} \Ve v \Ve_{W_A}.
\end{equation}
We now observe that, since $u_n\to u$ in $W_A$, we have (up to a subsequence) $u_n\to u$ and 
$\nabla u_n \to \nabla u$ pointwise a.e., $A|u_n|^p \le f$ and $B|u_n|^q\le g$ 
(by Proposition~\ref{sobolev.prop}) for some functions $f,g \in L^1(\rn)$, uniformly in $n$.
The limit in \eqref{limito} then follows from estimates \eqref{est1}-\eqref{est3} by dominated 
convergence, up to a subsequence. Since the previous argument can be applied to any subsequence of 
$(u_n)$, this completes the proof.
\hfill$\Box$

\end{document}